\documentclass[reqno]{amsart}
\usepackage{pictex}
\usepackage{mathrsfs}
\usepackage{hyperref}
\usepackage{color}
\usepackage{graphicx}

\begin{document}

%
%

\def\labelenumi{(\theenumi)}

\newtheorem{thm}{Theorem}[section]
\newtheorem{lem}[thm]{Lemma}
\newtheorem{conj}[thm]{Conjecture}
\newtheorem{cor}[thm]{Corollary}
\newtheorem{add}[thm]{Addendum}
\newtheorem{prop}[thm]{Proposition}
\theoremstyle{definition}
\newtheorem{defn}[thm]{Definition}
\theoremstyle{remark}
\newtheorem{rmk}[thm]{Remark}
\newtheorem{example}[thm]{{\bf Example}}
\newtheorem{Question}[thm]{{ \bf Question}}

\newcommand{\SLTwoC}{\mathrm{SL}(2,{\mathbb C})}
\newcommand{\SLTwoR}{\mathrm{SL}(2,{\mathbb R})}
\newcommand{\SUTwo}{\mathrm{SU}(2)}
\newcommand{\PSLTwoC}{\mathrm{PSL}(2,{\mathbb C})}
\newcommand{\GLTwoZ}{\mathrm{GL}(2,{\mathbb Z})}
\newcommand{\PGLTwoZ}{\mathrm{PGL}(2,{\mathbb Z})}
\newcommand{\GLTwoC}{\mathrm{GL}(2,{\mathbb C})}
\newcommand{\PSLTwoR}{\mathrm{PSL}(2,{\mathbb R})}
\newcommand{\PGLTwoR}{\mathrm{PGL}(2,{\mathbb R})}
\newcommand{\GLTwoR}{\mathrm{GL}(2,{\mathbb R})}
\newcommand{\PSLTwoZ}{\mathrm{PSL}(2,{\mathbb Z})}
\newcommand{\SLTwoZ}{\mathrm{SL}(2,{\mathbb Z})}
\newcommand{\nnn}{\noindent}

\newcommand{\bb}{{\mathbf{b}}}
\newcommand{\dd}{{\mathbf{d}}}
\newcommand{\cc}{{\mathbf{c}}}
\newcommand{\ff}{{\mathbf{f}}}
\renewcommand{\gg}{{\mathbf{g}}}
\newcommand{\nn}{{\mathbf n}}
\newcommand{\qq}{{\mathbf q}}
\newcommand{\uu}{{\mathbf u}}
\newcommand{\vv}{{\mathbf v}}
\newcommand{\ww}{{\mathbf w}}
\newcommand{\pp}{{\mathbf p}}
\newcommand{\HH}{{\mathbb H}}
\newcommand{\EE}{{\mathbb E}}
\newcommand{\CC}{{\mathbb C}}
\newcommand{\RR}{{\mathbb R}}
\newcommand{\QQ}{{\mathbb Q}}
\newcommand{\ZZ}{{\mathbb Z}}
\newcommand{\tr}{{\rm tr\, }}

\def\square{\hfill${\vcenter{\vbox{\hrule height.4pt \hbox{\vrule width.4pt
height7pt \kern7pt \vrule width.4pt} \hrule height.4pt}}}$}

\newenvironment{pf}{\noindent {\sl Proof.}\quad}{\square \vskip 12pt}

\title[Continuous expansions in Euclidean space]{On continuous expansions of configurations of points  in Euclidean space}

\author[H. Cheng, S. P. Tan, Y. Zheng]{Holun Cheng, Ser Peow Tan and Yidan Zheng}
\address{Department of Computer Science \\ National University of Singapore \\
 Singapore }
\email{hcheng@comp.nus.edu}%
\address{Department of Mathematics \\ National University of Singapore  \\ Singapore 119076 }
\email{mattansp@nus.edu.sg}%
\address{Department of Mathematics \\ National University of Singapore  \\ Singapore 119076 }
\email{lindelof.z@gmail.com}%



\thanks{ The
second author is partially supported by the National University of
Singapore academic research grant R-146-000-133-112}

%
%

\begin{abstract}
For any two configurations of ordered points $\pp=(\pp_{1},\cdots,\pp_{N})$
and $\qq=(\qq_{1},\cdots,\qq_{N})$ in Euclidean space $\EE^d$ such that
 $\qq$ is an expansion of $\pp$, there exists a continuous expansion from $\pp$ to $\qq$ in dimension $2d$;  Bezdek and Connelly used this to prove the Kneser-Poulsen conjecture for the planar case. In this paper, we show that this construction is optimal in the sense that for any $d \ge 2$ there exists configurations of $(d+1)^2$ points
$\pp$ and $\qq$ in $\EE^d$ such that $\qq$ is an expansion of $\pp$ but there is no continuous expansion from $\pp$ to $\qq$ in dimension less than $2d$. The techniques used in our proof are completely elementary.

\end{abstract}

\maketitle


\section{Introduction and statement of results.}\label{s:intro}
Let $\EE^d$ be the Euclidean space of dimension $d \ge 2$, where we identify and represent the points of $\EE^d$ by their position vectors. $\EE^d$ is endowed with the standard inner product $\uu\cdot \vv$ and norm $|\uu|=\sqrt{\uu.\uu}$.

Suppose that $d<f$, then $\EE^f \cong \EE^d \times \EE^{f-d}$ and we have  the standard projections $\pi_1: \EE^f \rightarrow \EE^d$ and $\pi_2: \EE^f \rightarrow \EE^{f-d}$ given by
$$\pi_1(u_1,\ldots,u_f)=(u_1, \ldots,u_d), \qquad \pi_2(u_1,\ldots,u_f)=(u_{d+1}, \ldots,u_{f}),$$
and the standard inclusion  $\iota: \EE^d \rightarrow \EE^f$ given by $$\iota(\uu)=\iota(u_1,\ldots,u_d)=(u_1,\ldots,u_d,0\ldots,0).$$
Note that $\pi_1 \circ \iota=id$ on $\EE^d$,  and for $\uu, \vv \in \EE^f$,
\begin{eqnarray}
    \uu &=& (\pi_1(\uu), \pi_2(\uu)), \\
  \uu \cdot\vv &=& \pi_1(\uu)\cdot \pi_1(\vv)+\pi_2(\uu)\cdot \pi_2(\vv), \label{eqn:dotprod}\\
 | \uu|^2 &=& | \pi_1(\uu)|^2+| \pi_2(\uu) |^2. \label{eqn:norms}
\end{eqnarray}

\medskip

Let $\pp=(\pp_{1},\cdots,\pp_{N})$ and $\qq=(\qq_{1},\cdots,\qq_{N})$ be two
configurations of $N$ ordered points in $\EE^{d}$, where $\pp_{i},\qq_{i}\in\EE^{d}$
for $i=1,\ldots,N$, and suppose $f > d$.
\begin{defn}(Expansions in $\EE^d$)
$\qq$ is an expansion of $\pp$ if  $$| \pp_i-\pp_j|\le| \qq_i-\qq_j|, \qquad 1 \le i<j \le N.$$
\end{defn}

\begin{defn}(Continuous expansions in $\EE^f$)
We say that there is a continuous expansion from $\pp$ to $\qq$ in $\EE^f$ if there exists a family of continuous functions (\emph{continuous motions})
$$\ff_i:[0,1]\longrightarrow\mathbb{E}^{f}, \quad i=1, \ldots N$$
such that for $1 \le i <j \le N$ and $0 \le t_1<t_2 \le 1$,
\begin{enumerate}
\item $\ff_i(0)=\iota(\pp_i)$, $\ff_i(1)=\iota(\qq_i)$;
\item $| \ff_{i}(t_1)-\ff_{j}(t_1)| \le | \ff_{i}(t_2)-\ff_{j}(t_2)|$.

\end{enumerate}
\end{defn}

 Note that if there is a continuous expansion from $\pp$ to $\qq$ in $\EE^f$,
then $\qq$ is necessarily an expansion of $\pp$ in $\EE^d$, but an expansion may not admit a continuous expansion in the same or a higher dimension. The following result by R. Alexander \cite{Alex} shows that any expansion admits a continuous expansion in twice the dimension.

\begin{thm}\label{thm:continuous}\cite{Alex}, see also \cite{BezCon}.
Suppose that $\qq=(\qq_1, \ldots \qq_N)$ is an expansion of $\pp=(\pp_1, \ldots \pp_N)$ in $\EE^d$. Then  the family of functions $\ff_i:[0,1] \longrightarrow \EE^{2d}, \quad i=1, \ldots, N$, given by
$$\ff_{i}(t)=\left(\frac{\pp_{i}+\qq_{i}}{2}+(cos\pi t)\frac{\pp_{i}-\qq_{i}}{2}, \,(sin\pi t)\frac{\pp_{i}-\qq_{i}}{2}\right)$$
is a continuous expansion from $\pp$ to $\qq$ in $\EE^{2d}$.
\end{thm}

\begin{proof}
We reproduce the proof here for completeness. Clearly, $\ff_i$ is continuous for $i=1, \ldots, N$ and $\ff_i(0)=\iota(\pp_i)$, $\ff_i(1)=\iota(\qq_i)$. Expanding,
$4| \ff_{i}(t)-\ff_{j}(t)|^{2}$
$$=|(\pp_{i}-\pp_{j})-(\qq_{i}-\qq_{j})|^{2}+|(\pp_{i}-\pp_{j})+(\qq_{i}-\qq_{j})|^{2}+2(\cos\pi t)(| \pp_{i}-\pp_{j}|^{2}-| \qq_{i}-\qq_{j}|^{2}).$$
Since $\qq$ is an expansion of $\pp$, $| \pp_{i}-\pp_{j}|^{2}-| \qq_{i}-\qq_{j}|^{2}\le0$
for all $ i \neq j$. Therefore $| \ff_{i}(t)-\ff_{j}(t)|$ is non-decreasing on $[0,1]$.
\end{proof}

Bezdek and Connelly  used the above in \cite{BezCon}, together with results of Csik\'os \cite{Csi} to prove the Kneser-Poulsen conjecture \cite{Kne} for the plane. More specifically, they showed that if there is a piecewise analytic expansion from $\pp$ to $\qq$ in dimension $d+2$, then the Kneser-Poulsen conjecture holds for balls centered at $\pp$ and $\qq$, that is, the volume of the union of the balls $B(\pp_i,r_i)$ is less than or equal to the volume of the union of the balls $B(\qq_i,r_i)$, where $r_i>0$. Similarly, the same method shows that the conjecture holds if the number of balls $N \le d+3$, generalizing a result of Gromov in \cite{Gro}. This raises the question, as pointed out in \cite{BezCon}, of whether it is possible to find continuous expansions in dimensions less than $2d$ for all expansions $\qq$ of $\pp$ in dimension $d$. If so, then the approach of Bezdek and Connelly can be applied to prove the Kneser-Poulsen conjecture in more general settings. Our main result is a negative answer to this question, specifically, we have:

\begin{thm}\label{thm:main}(Main Theorem)
There exists  configurations  $\pp=(\pp_1, \ldots \pp_N)$ in $\EE^d$ with expansions $\qq=(\qq_1, \ldots \qq_N)$, where $N=(d+1)^2$, which do not admit continuous expansions in dimensions less than $2d$.
\end{thm}

\noindent {\it Remark:} The example we construct is in fact the same as that constructed independently by Belk and Connelly in \cite{BelCon}, and in both cases, based on the example constructed in \cite{BezCon} for the planar case. However, our proof is more elementary and uses only basic linear algebra and some simple rigidity results. Indeed, our proof shows that away from the endpoints, any continuous expansion cannot be embedded into dimension less than $2d$ at any time $t \in (0,1)$.

\bigskip

The configurations $\pp$ and $\qq$ are built from the $(d+1)$
vertices $\vv_0, \ldots, \vv_d$ of the regular $d$-simplex $\sigma_d
\subset \EE^d$, together with the vertices  of the inward and
outward flaps associated to the faces of $\sigma_d$, specifically,
each face $F^i$ ($i=0. \ldots,d$) of $\sigma_d$ may be pushed
orthogonally towards or away from the center of $\sigma_d$ by a
distance $s>0$, to obtain flaps $F^i_{inw}$ and $F^i_{out}$
respectively (note that Belk and Connelly had a slightly different
definition for flaps in \cite{BelCon}). The configuration $\pp$
consists of the vertices of $\sigma_d$ and of the inward flaps
$F^i_{inw}$ and the configuration $\qq$ consists of the
corresponding vertices of $\sigma_d$ and of the outward flaps
$F^i_{out}$. Note that each flap has $d$ vertices so that $\pp$ and
$\qq$ consists of $(d+1)^2$ points. The rest of the paper will be
devoted to explaining this construction (\S \ref{s:simplex}),
showing that $\qq$ is an expansion of $\pp$ (\S \ref{s:expansion}),
and proving that there is no continuous expansion from $\pp$ to
$\qq$ in $\EE^f$ for $f <2d$ (\S \ref{s:proof}).

\vskip 6pt

\noindent {\it Acknowledgements.} This work arose from an undergraduate honors project of the third author under the
supervision of the first and second authors. The authors are grateful to Jean-Marc
Schlenker for helpful conversations, and also for bringing their attention to
\cite{BelCon} arising from his correspondence with R. Connelly.


\section{Regular simplices with flaps}\label{s:simplex}
Let $\sigma :=\sigma_d \subset \EE^d$ be the regular simplex with vertices $\uu_i$, $i=0, \ldots, d$ and center at the origin $O$ such  that $| \uu_i| =1$ for all $i$ (see figure 1). Then
\begin{equation}\label{eqn:innerproductofnorms}
   \uu_i \cdot \uu_j=-\frac{1}{d}, \quad i \neq j
\end{equation}

\begin{figure}[htbp]
    \centering
        \includegraphics[width=0.70\textwidth]{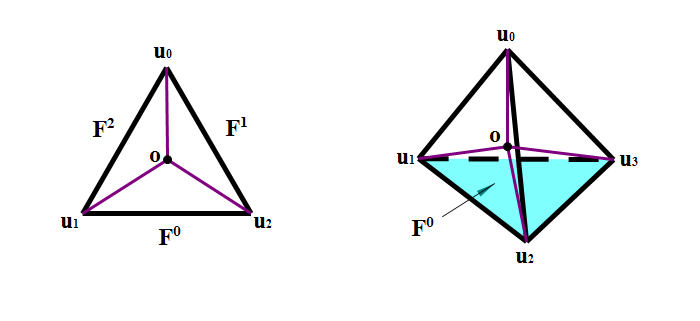}
    \caption{The simplex $\sigma_2$ and $\sigma_3$}
\end{figure}

see for example Coxeter \cite{Cox}, or Parks and Wills \cite{ParWil}
for an elementary proof. Denote by $F^i$ the face of $\sigma$ which
does not contain the vertex $\uu_i$. Then the norm of $F^i$, the
outward facing unit normal $\nn_i$ to $F^i$ is the vector $-\uu_i$.

Fix $s>0$. For each face $F^i$, $i=0, \ldots, d$, define the outward
$i$th flap of depth $s$ to be $F^i$ translated by $s\nn_i=-\uu_i$,
that is,
$$F^i_{out}:=F^i -s\uu_i.$$
Similarly, the inward $i$th flap of depth $s$ is given by
$$F^i_{inw}:=F^i +s\uu_i.$$
Each flap has $d$ vertices and if we denote the vertices of
$F^i_{out}$ by $\cc^i_j$ and those of $F^i_{inw}$ by $\bb^i_j$,
where $j\neq i$ (see figure 2 for the case when $d=2$ and $3$), then
we have, for $i,j\in\{0, \ldots, d\}$, $i \neq j$,
\begin{eqnarray}
  \cc^i_j &=& \uu_j+s\nn_i = \uu_j-s\uu_i \\
  \bb^i_j &=& \uu_j-s\nn_i = \uu_j+s\uu_i
\end{eqnarray}

\begin{figure}[htbp]
    \centering
        \includegraphics[width=0.80\textwidth]{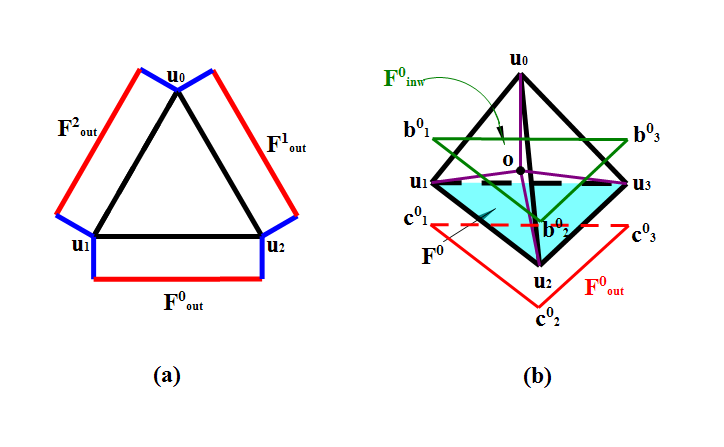}
    \caption{(a) $\sigma_2$ with outward flaps  (b) $\sigma_3$ with the inward and outward flaps $F^0_{inw}$ and $F^0_{out}$}
\end{figure}

 The configurations $\pp$ and $\qq$ we are interested in consists of the vertices of the regular simplex with inward and outward flaps respectively, defined by
\begin{equation}\label{eqn:pandq}
   \pp=\{\uu_i\} \cup \{\bb^i_j\}, \quad \qq =\{\uu_i\} \cup \{\cc^i_j\}, \quad i,j \in \{0,1,\ldots,d\},\quad i \neq j
\end{equation}
where $\pp$ and $\qq$ are ordered so that the  correspondence between the elements from the indexing is preserved. We have:

\begin{thm}\label{thm:main2}
Suppose that $\pp$ and $\qq$ are configurations in $\EE^d$ consisting of the vertices of the regular simplex with inward and outward flaps defined as in (\ref{eqn:pandq}). Then
\begin{enumerate}
  \item [(a)] $\qq$ is an expansion of $\pp$ in $\EE^d$;
  \item [(b)] there does not exist a continuous expansion from $\pp$ to $\qq$ in $\EE^f$ for $f <2d$.
\end{enumerate}
\end{thm}
We will prove (a) in the next section and (b) in the following
section. We note that although (a) was claimed in \cite{BelCon}, no
proof was given, we give a proof here for completeness. Also our
proof of (b) is independent of, and more elementary than that given
in \cite{BelCon}.

\section{Proof that $\qq$ is an expansion of $\pp$}\label{s:expansion}

We only need to consider the distances between vertices on $\sigma$ and vertices on the flaps, or between vertices on the flaps.
In the first case, we have $$| \uu_k- \bb^i_j| =| \uu_k- \cc^i_j|, \quad \hbox{if}\quad k \neq i,$$ since $\uu_k \subset F^i$, and
$$| \uu_i- \bb^i_j| < | \uu_i- \cc^i_j|$$  since by reflecting on the face $F^i$, we see there is a broken path from $\uu_i$
to $\bb^i_j$ of length $| \uu_k- \cc^i_j|$. The argument works if we
replace $\sigma$ by any simplex.

In the second case, we have, for $i \neq j$, $k \neq l$,
\begin{eqnarray*}
  | \bb_j^i-\bb_l^k|^2 &=& | (\uu_j+s\uu_i)-(\uu_l+s\uu_k)  |^2 \\
  ~ &=& | \uu_j-\uu_l |^2+2s(\uu_j-\uu_l)\cdot(\uu_i-\uu_k)+s^2| \uu_i-\uu_k |^2 \\
   | \cc_j^i-\cc_l^k|^2 &=& | \uu_j-\uu_l |^2-2s(\uu_j-\uu_l)\cdot(\uu_i-\uu_k)+s^2| \uu_i-\uu_k |^2 \\
  \Longrightarrow \quad | \bb_j^i-\bb_l^k|^2 &-& | \cc_j^i-\cc_l^k|^2=  4s(\uu_j-\uu_l)\cdot(\uu_i-\uu_k)
\end{eqnarray*}
If $i=k$, or $j=l$, or $i,j,k,l$ are all distinct, then
$4s(\uu_j-\uu_l)\cdot(\uu_i-\uu_k)=0$ by
(\ref{eqn:innerproductofnorms}) so that $$| \bb_j^i-\bb_l^k|= |
\cc_j^i-\cc_l^k|.$$
If $i=l$ or $j=k$, then
$$4s(\uu_j-\uu_l)\cdot(\uu_i-\uu_k)=4s(\frac{1}{d}-1)<0$$ by (\ref{eqn:innerproductofnorms}), hence in all cases,
$$| \bb_j^i-\bb_l^k|\leq  | \cc_j^i-\cc_l^k|.$$ \qed

\noindent {\it Remark:} In the case where we start with any simplex
instead of $\sigma_d$, then $$| \bb_j^i-\bb_l^k|^2 - |
\cc_j^i-\cc_l^k|^2= 4s(\uu_j-\uu_l)\cdot(\nn_k-\nn_i).$$ Again, if
$i=k$, or $j=l$, or $i,j,k,l$ are all distinct, then
$4s(\uu_j-\uu_l)\cdot(\nn_k-\nn_i)=0$, and if $i=l$ or $j=k$, then
$4s(\uu_j-\uu_l)\cdot(\nn_k-\nn_i)<0$, so Theorem \ref{thm:main2}(a)
holds if we replace the regular simplex with any simplex.

\section{ Proof that there is no continuous expansion in dimension $<2d$}\label{s:proof}
The main tools we use are some basic linear algebra as described in \S \ref{s:intro}, and the fact that the configurations $\pp$ and $\qq$ contain several sub-configurations which are rigid under continuous expansion since the pair-wise distances are preserved in the sub-configurations. We first outline the strategy of our proof, note that it suffices to show that there is no continuous expansion in dimension $2d-1$.
\begin{enumerate}
    \item[(I)] We will assume for a contradiction that there exists a continuous expansion from $\pp$ to $\qq$ in $\EE^{2d-1}$;
  \item [(II)] we construct for each face $F^k$ a displacement vector function $$\dd_k:[0,1] \longrightarrow \EE^{2d-1} \cong \EE^d \times \EE^{d-1};$$
  such  that $\dd_k(t)$ is orthogonal to $F^k$ and $| \dd_k(t)| =s$ for all $t \in [0,1]$;
  \item [(III)] show that there is some $t_0 \in [0,1]$ such that the projection $\pi_2(\dd_k(t_0))$ to $\EE^{d-1}$ is non-zero for all $k\in \{0,1, \ldots,d\}$;
  \item [(IV)] show that the set $\{\ww_k=\pi_2(\dd_k(t_0))\}\subset \EE^{d-1}$ consists of pairwise obtuse vectors;
  \item [(V)] show that this is not possible to give the required contradiction.
\end{enumerate}

\vskip 15pt

\noindent (I) Consider $\EE^{2d-1}\cong \EE^d \times \EE^{d-1}$  and define the projections $\pi_1:\EE^{2d-1} \rightarrow \EE^d$ and $\pi_2:\EE^{2d-1} \rightarrow \EE^{d-1}$ and the inclusion $\iota: \EE^d \rightarrow \EE^{2d-1}$  as in \S \ref{s:intro}.

Suppose that there is a continuous expansion from $\pp$ to $\qq$ in $\EE^{2d-1} \cong \EE^d \times \EE^{d-1}$. Let $\ff_k,~~ \gg_j^i:[0,1]\rightarrow \EE^{2d-1}$, $i,j,k \in \{0, \ldots, d\}$, $i \neq j$,  be the continuous motions of $\uu_k$ and $\bb_j^i$ respectively which define the continuous expansion from $\pp$ to $\qq$. Since $\sigma_d$ is rigid, we may assume without loss of generality that $\uu_k$ remains stationary throughout the motion, that is
\begin{equation}\label{eqn:ffkt}
   \ff_k(t) \equiv \iota(\uu_k), \quad k=0, \ldots, d.
\end{equation}
We also have
\begin{equation}\label{eqn:ggji}
    \gg_j^i(0)=\iota(\bb_j^i)=\iota(\uu_j+s\uu_i), ~~ \gg_j^i(1)=\iota(\cc_j^i)=\iota(\uu_j-s\uu_i).
\end{equation}

\noindent (II) We will need the following:

\begin{prop}\label{prop:parallelogram}
Suppose that $(\uu_1,\uu_2, \uu_3,\uu_4), (\vv_1,\vv_2,\vv_3,\vv_4) \subset \EE^n$ are configurations such that
\begin{equation}\label{eqn:parallelogram}
   |\uu_i-\uu_j|=|\vv_i-\vv_j| \quad \hbox{ for all} \quad i \neq j.
\end{equation}
 If $(\uu_1,\uu_2, \uu_3,\uu_4)$ is a parallelogram, then $(\vv_1,\vv_2,\vv_3,\vv_4)$ is also a parallelogram  and $(\vv_1, \vv_2, \vv_3, \vv_4)\cong (\uu_1,\uu_2, \uu_3,\uu_4)$.
\end{prop}

\begin{proof} Let $\ww$ and $\ww'$ be the midpoints of $(\uu_2,\uu_4)$ and $(\vv_2,\vv_4)$ respectively. We have $\triangle(\uu_1,\uu_2, \uu_4)\cong \triangle(\vv_1,\vv_2, \vv_4)$, hence  $|\ww-\uu_1|=|\ww'-\vv_1|$ (see figure 3). Similarly, $\triangle(\uu_2,\uu_3, \uu_4)\cong \triangle(\vv_2,\vv_3, \vv_4)$, so $|\uu_3-\ww|=|\vv_3-\ww'|$. Also, by (\ref{eqn:parallelogram}) $|\vv_3-\vv_1|=|\uu_3-\uu_1|$  and  since $(\uu_1,\uu_2, \uu_3,\uu_4)$ is a parallelogram, $|\uu_3-\uu_1|=|\uu_3-\ww|+|\ww-\uu_1|$. Hence
$$|\vv_3-\vv_1|=|\uu_3-\uu_1|=|\uu_3-\ww|+|\ww-\uu_1|=|\vv_3-\ww'|+|\ww'-\vv_1|.$$

\begin{figure}[htbp]
    \centering
        \includegraphics[width=0.80\textwidth]{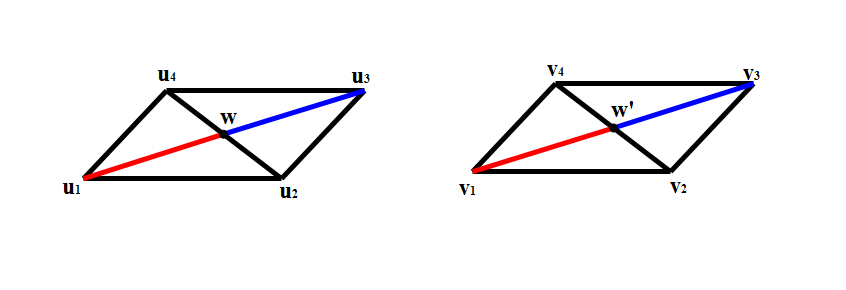}
    \caption{The configurations ($\uu_1$,$\uu_2$,$\uu_3$,$\uu_4$) and ($\vv_1$,$\vv_2$,$\vv_3$,$\vv_4$)}
    \label{fig:figure_3}
\end{figure}

Hence, $\vv_1, \ww'$ and $\vv_3$ are collinear, and  $(\vv_1, \vv_2, \vv_3, \vv_4)$ lies on a plane with the diagonal from $\vv_1$ to $\vv_3$ bisecting the diagonal from $\vv_2$ to $\vv_4$. A similar argument shows that the diagonal from $\vv_2$ to $\vv_4$ bisects the diagonal from $\vv_1$ to $\vv_3$, so that $(\vv_1, \vv_2, \vv_3, \vv_4)$ is a parallelogram. Now (\ref{eqn:parallelogram}) implies  $(\vv_1, \vv_2, \vv_3, \vv_4)\cong (\uu_1,\uu_2, \uu_3,\uu_4)$.

\end{proof}

\noindent Now, for distinct $i,j,k \in \{0, \ldots, d\}$, consider the continuous family of configurations $(\ff_i(t), \ff_j(t), \gg_j^k(t), \gg_i^k(t))$, $t \in [0,1]$. By assumption, this is a continuous expansion, but the pairwise distances between points in the initial configuration $$(\ff_i(0), \ff_j(0), \gg_j^k(0), \gg_i^k(0))=((\iota(\uu_i), \iota(\uu_j), \iota(\uu_j+s\uu_k), \iota(\uu_i+s\uu_k))$$ and those of the final configuration $$(\ff_i(1), \ff_j(1), \gg_j^k(1), \gg_i^k(1))=((\iota(\uu_i), \iota(\uu_j), \iota(\uu_j-s\uu_k), \iota(\uu_i-s\uu_k))$$ are equal since they form congruent rectangles. Since the initial configuration describes a rectangle, it follows from proposition \ref{prop:parallelogram} that all intermediate configurations are congruent rectangles. Hence,
$$\gg_j^k(t)-\iota(\uu_j)=\gg_i^k(t)-\iota(\uu_i), \quad \forall~~ i \neq j \neq k \neq i.$$
We can define $\dd_k(t): [0,1] \rightarrow \EE^{2d-1}$ by
$$\dd_k(t):=\gg_j^k(t)-\iota(\uu_j), \quad \hbox{for any} \quad j \neq k,$$ then
$$|\dd_k(t)|=|\gg_j^k(t)-\iota(\uu_j)|=|\gg_j^k(0)-\iota(\uu_j)|=s$$
and $\dd_k(t) \cdot (\iota(\uu_j-\uu_i))=0$ for all $i \neq j \neq k \neq i$, hence, $\dd_k(t)$ is orthogonal to $\iota(F^k)$, since $\{(\iota(\uu_j-\uu_i))\}$, $i \neq j \neq k \neq i$ spans $\iota(F^k)$.
\bigskip

\bigskip

\noindent (III) For $k=0, \ldots, d$, let $$\pi_1(\dd_k(t)):=\vv_k(t)\in \EE^d, \quad \pi_2(\dd_k(t)):=\ww_k(t) \in \EE^{d-1}$$ so that $\dd_k(t)=(\vv_k(t), \ww_k(t))$. Since $\dd_k(t).\iota(\uu_i-\uu_j)=\vv_k(t).(\uu_i-\uu_j)=0$ for all $i \neq j \neq k \neq i$, $\vv_k(t)$ is orthogonal to $F^k \subset \EE^d$,  so  $\vv_k(t)=a_k(t)\uu_k$, $a_k(t) \in \RR$, and furthermore, $|a_k(t)| \le s$ since $|\vv_k(t)|^2+|\ww_k(t)|^2=|\dd_k(t)|^2=s^2$ by (\ref{eqn:norms}). By the intermediate value theorem, since $a_k(0)=s$ and $a_k(1)=-s$, $a_k(t)$ takes all values in $[-s,s]$, so in particular, there exists some $t_0 \in [0,1]$ such that $a_k(t_0)=0$ , so that $\vv_k(t_0)={\mathbf 0}$.
Hence $|\ww_k(t_0)|^2=s^2$, in particular, $\ww_k(t_0) \neq {\mathbf 0}$ (in fact, we only need that $|a_k(t_0)|<s$ to get $\ww_k(t_0) \neq {\mathbf 0}$).

\medskip
Now for $i \neq j \neq k \neq i$, we have $\triangle(\iota(\uu_i),\gg_i^j(0),\gg_i^k(0)) \cong \triangle(\iota(\uu_i),\gg_i^j(1),\gg_i^k(1))$ since
$$\gg_i^j(0)-\iota(\uu_i)=\iota(s\uu_j), \quad \gg_i^k(0)-\iota(\uu_i)=\iota(s\uu_k),$$
$$ \gg_i^j(1)-\iota(\uu_i)=\iota(-s\uu_j), \quad \gg_i^k(1)-\iota(\uu_i)=\iota(-s\uu_k),$$
so all the triangles $ \triangle (\iota(\uu_i),\gg_i^j(t),\gg_i^k(t))$, $t \in [0,1]$ are congruent. In particular,
\begin{equation}\label{eqn:dkdotdj}
    (\gg_i^k(t)-\iota(\uu_i))\cdot (\gg_i^j(t)-\iota(\uu_i))=\dd_k(t) \cdot \dd_j(t)=\dd_k(0) \cdot \dd_j(0)=s\uu_k \cdot s\uu_j=-\frac{s^2}{d}
\end{equation}
 for all $t \in [0,1]$ by (\ref{eqn:innerproductofnorms}).
Now using $\vv_k(t_0)={\mathbf 0}$ and applying  (\ref{eqn:dotprod}) to (\ref{eqn:dkdotdj}) gives,
\begin{equation}\label{eqn:wkdotwj}
    -\frac{s^2}{d}=\dd_k(t_0)\cdot \dd_j(t_0)=\vv_k(t_0) \cdot \vv_j(t_0)+\ww_k(t_0)\cdot \ww_j(t_0)=\ww_k(t_0)\cdot \ww_j(t_0)
\end{equation} for all $j \neq k$. In particular, we see that $\ww_j(t_0) \neq {\mathbf 0}$ for all $j=0, \ldots, d$ (again, we really only need that $|a_k(t_0)|<s$ to obtain this conclusion).

\bigskip

\noindent (IV) We need to show that $\ww_i(t_0)\cdot \ww_j(t_0)<0$ for all distinct $i, j \in \{0, \ldots, d\}$.
Recall that $\dd_i(t)=(\vv_i(t), \ww_i(t))=(a_i(t)\uu_i, \ww_i(t))$. Since $\ww_i(t_0) \neq {\mathbf 0}$ and by (\ref{eqn:norms}) $$s^2=|\dd_i(t_0)|^2=|\vv_i(t_0)|^2+|\ww_i(t_0)|^2=|a_i(t_0)|^2+|\ww_i(t_0)|^2$$
we have
\begin{equation}\label{eqn:aitlessthans}
    -s< a_i(t_0)<s, \quad \hbox{for all} \quad i=0, \ldots,d.
\end{equation}

 Now by (\ref{eqn:innerproductofnorms}), for $i \neq j$,
$$\dd_i(t_0) \cdot \dd_j(t_0)=\vv_i(t_0) \cdot \vv_j(t_0) +\ww_i(t_0) \cdot \ww_j(t_0).$$
$\dd_i(t_0) \cdot \dd_j(t_0)= \dd_i(0)\cdot \dd_j(0)=-\frac{s^2}{d}$ and $$\vv_i(t_0) \cdot \vv_j(t_0)=a_i(t_0)a_j(t_0)\uu_i\cdot \uu_j=-\frac{a_i(t_0)a_j(t_0)}{d},$$ where by (\ref{eqn:aitlessthans}), $|\vv_i(t_0) \cdot \vv_j(t_0)|< \frac{s^2}{d}$. It follows that $\ww_i(t_0)\cdot \ww_j(t_0)<0$ for all distinct $i, j \in \{0,\ldots, d\}$.

\bigskip

\noindent {\it Remark:} In proving the conclusion in (IV) holds, we
only really require that the outward normals $\nn_i$, $i=0, \ldots,
d$ of $\sigma_d$ are pairwise obtuse, that is, $\nn_i \cdot \nn_j<0$
for all distinct $i,j \in \{0, \ldots, d\}$. Hence we may replace
the regular simplex with one for which the above holds.

\bigskip

\noindent (V) Recall that $\uu_1, \uu_2 \in \EE^n$ are obtuse if $\uu_1 \cdot \uu_2 <0$. The lemma below states that we cannot have a collection of $n+2$  pairwise  obtuse vectors in $\EE^n$.
\begin{lem}\label{lem:obtuse} For any set $\{\uu_1, \ldots, \uu_{n+2}\}$ of $n+2$ vectors in $\EE^n$, $\uu_i\cdot \uu_j \ge 0$ for some $i \neq j$, that is, the vectors cannot be all pairwise obtuse.

\end{lem}

\begin{proof} We prove by induction on the dimension $n$. The result is clearly true when $n=1$ since for any 3 vectors $\uu_1, \uu_2, \uu_3 \in \EE^1$, either at least one of the vectors is $\mathbf 0$, or two are in the same direction so have positive dot product. Assume the lemma is true for $n$ and suppose for a contradiction that there exists $\uu_1, \ldots, \uu_{n+3} \in \EE^{n+1}$ that are all pairwise obtuse. Without loss of generality, we may assume that none of $\uu_i$ are zero, and that $\uu_{n+3}=(-1,0, \ldots, 0)$.
Write $\EE^{n+1}\cong \EE^1 \times \EE^{n}$ and consider the projections $\pi_1:\EE^{n+1} \rightarrow \EE^1$ and $\pi_2:\EE^{n+1} \rightarrow \EE^n$  respectively as in \S \ref{s:intro}. For $i=1, \ldots, n+2$, let $\vv_i:=\pi_1(\uu_i) \in \EE^1 \cong \RR$, $\ww_i:=\pi_2(\uu_i) \in \EE^n$, see figure 4. Note that $\vv_i >0$ since $\uu_i \cdot \uu_{n+3}<0$, so $\vv_i \cdot \vv_j>0$ for $i,j \in \{1, \ldots, n+2\}$.
 Then we have, from (\ref{eqn:dotprod}), for distinct $i, j \in \{1, \ldots, n+2\}$,
  $$\uu_i \cdot \uu_j=\vv_i\cdot \vv_j +\ww_i \cdot \ww_j.$$

  \begin{figure}[htbp]
    \centering
        \includegraphics[width=1\textwidth]{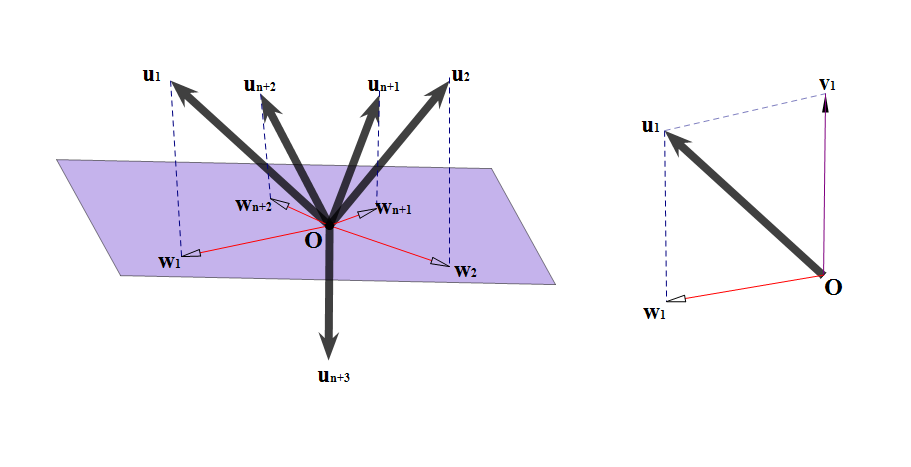}
    \caption{Projection of $\EE^{n+1}$ vectors into $\EE^{n}$ space }
    \label{fig:figure_4}
\end{figure}

 By assumption, $\uu_i \cdot \uu_j <0$, and $\vv_i\cdot \vv_j >0$ from the above, so
  $$\ww_i\cdot \ww_j <0.$$
  Hence $\{\ww_1, \ldots, \ww_{n+2}\}$ is a collection of pairwise obtuse vectors in $\EE^{n}$ contradicting the induction hypothesis.

\end{proof}

Applying  lemma \ref{lem:obtuse} to the set $\{\ww_0, \ww_1, \ldots, \ww_{d}\} \subset \EE^{d-1}$ in (IV) we get the required contradiction which concludes the proof of Theorem \ref{thm:main2} from which Theorem \ref{thm:main} follows. \qed

\bigskip

\noindent{\bf Concluding remarks.} The method of proof above works
if we construct $\pp$ and $\qq$ from any simplex in $\EE^d$ whose
pairwise norms are obtuse. It also shows that any intermediate
configuration in a continuous expansion from $\pp$ to $\qq$ cannot
be embedded in a space of dimension less than $2d$.
An interesting open question is, for each $d$, what is  the smallest number of points in the configurations $\pp$ and $\qq$  for which there is no continuous expansion in $\EE^{2d-1}$. We have shown that $N=(d+1)^2$ suffices, but this may not be optimal. Finally, it is also interesting to ask if we can find configurations $\pp$, and expansions $\qq$ of $\pp$ such that the continuous expansion given by Theorem \ref{thm:continuous} is essentially, up to some trivial motions, the only continuous expansion in dimension $2d$.

\end{document}